\documentclass[12pt,a4paper,oneside]{amsart}

\usepackage{amsfonts, amsmath, amssymb, amsthm, hyperref}
\usepackage{anysize}
\usepackage{graphics}

\newtheorem{thm}{Theorem}[section]
\newtheorem*{thm*}{Theorem}
\newtheorem{lem}[thm]{Lemma}

\newtheorem{cor}[thm]{Corollary}

\theoremstyle{definition}
\newtheorem{defn}[thm]{Definition}

\theoremstyle{remark}
\newtheorem{rem}[thm]{Remark}

\numberwithin{equation}{section}

\renewcommand{\epsilon}{\varepsilon}
\renewcommand{\phi}{\varphi}

\marginsize{2cm}{2cm}{2cm}{2cm}

\newcounter{ris}
\renewcommand{\r}{\refstepcounter{ris}%
                  Fig. \arabic{ris}}

\begin{document}
	\ifpdf
	\DeclareGraphicsExtensions{.mps, .jpg, .tif, .pdf}
	\else
	\DeclareGraphicsExtensions{.eps, .jpg, .mps}
	\fi

\title{Inscribing a regular octahedron into polytopes}

\author{Arseniy~Akopyan}
\address{Arseniy Akopyan, Institute for Information Transmission Problems RAS\\ 
			Bolshoy Karetny per. 19, Moscow, Russia 127994  \newline      
B.~N.~Delone International Laboratory ``Discrete and Computational Geometry'', P.~G~ Demidov Yaroslavl State University., Sovetskaya st. 14, Yaroslavl', Russia 150000}
\email{akopjan@gmail.com}

\author{Roman~Karasev}
\address{Roman Karasev, Dept. of Mathematics, Moscow Institute of Physics and Technology, Institutskiy per. 9, Dolgoprudny, Russia 141700
\newline
B.~N.~Delone International Laboratory ``Discrete and Computational Geometry'', P.~G~ Demidov Yaroslavl State University., Sovetskaya st. 14, Yaroslavl', Russia 150000}
\email{r\_n\_karasev@mail.ru}
\urladdr{http://www.rkarasev.ru/en/}
\thanks{The research of both authors is supported by the Dynasty Foundation, the President's of Russian Federation grant MD-352.2012.1, the Russian Foundation for Basic Research grant 10-01-00096, the Federal Program ``Scientific and scientific-pedagogical staff of innovative Russia'' 2009--2013, and the Russian government project 11.G34.31.0053. The research of A.V.~Akopyan is also supported by the Russian Foundation for Basic Research grant 11-01-00735 and the research of R.N.~Karasev is also supported by the Russian Foundation for Basic Research grant 10-01-00139.}
\maketitle

\begin{abstract}
We prove that any simple polytope (and some non-simple polytopes) in $\mathbb R^3$ admits a regular octahedron inscribed into its surface.
\end{abstract}


\section{Introduction}

The famous theorem of Schnirelmann \cite{shnirelman1944oncerain} (see also \cite{klee1991old}) asserts that for every smooth simple closed curve $\gamma$ in the plane there exists a square $Q$ such that all four vertices of $Q$ lie on $\gamma$.

In the thesis of Vladimir~Makeev~\cite{makeev2003univerally} the following theorem was proved (reproved and generalized for higher prime power dimensions in~\cite{karasev2009inscribing,karasev2010equipartitionbyfans}):

\begin{thm}
\label{thm:insrccross-3}
Let $H$ be a smooth surface embedded into $\mathbb R^3$ and homeomorphic to the sphere $S^2$. Let $C$ be some $\mathbb Z_3$-symmetric octahedron. Then there exists an octahedron $C'\subset\mathbb R^3$ similar to $C$ with all its vertices lying on $H$.
\end{thm}

\begin{rem}
Here by \emph{$\mathbb Z_3$-symmetric} we mean the following: The group $\mathbb Z_3\subset SO(3)$ acts on $C$ by cyclically permuting three pairs of its opposite vertices so that this action is extended to an action by isometries on the whole $\mathbb R^3$.

The word \emph{similar} here means equivalent up to a similarity transform with positive determinant, that is a composition of a proper rigid motion and a positive homothety. By \emph{inscribing} a polytope $P$ into a surface $H$ we will always mean finding its similar copy such that all vertices of $P$ lie on $H$.
\end{rem}

It is known (see~\cite{klee1991old,matschke2009square,paklectures} for example) that squares in the plane can be also inscribed into any polygonal simple curve; the approximation by smooth curves and going to the limit works well in this case. Informally, the key feature here is that if you look at the square from some direction in the plane then you do not see one of its vertices. The problem of inscribing a square into any continuous simple closed curve in the plane remains unsolved, the approximation by smooth curves seems to be insufficient in this case, see~\cite{matschke2009square} for further information about this.

In $\mathbb R^3$ the situation is different even for regular octahedra: One can see all the vertices of an octahedron from some directions. Thus we have to be careful when going to the limit and this is the main content of this paper. Note also that in the plane there exists a direct proof~\cite{pak2008discrete} of the Schnirelmann theorem for polygonal curves, while in this paper we cannot avoid using the smooth case and going to the limit. 

The main result of this paper is:

\begin{thm}
\label{thm:simple-inscr}
Suppose $P$ is a simple convex polytope in $\mathbb R^3$. Then there exists a regular octahedron inscribed into the surface $\partial P$.
\end{thm}

\begin{rem}
The case of non-simple polytopes is considered in Section~\ref{sec:non-simple-sec}. In this case we are able to establish the existence of an inscribed octahedron under certain additional assumptions on its solid angles (see Theorem~\ref{thm:non-thm:simple-inscr}).
\end{rem}

\begin{rem}
Unlike the results for inscribing into smooth surfaces~\cite{karasev2009inscribing,karasev2010equipartitionbyfans}, here we only consider inscribing of a regular octahedron, leaving out the case of inscribing an octahedron similar to a given $\mathbb Z_3$-symmetric octahedron. This case seems to be out of reach of the methods presented here.
\end{rem}

{\bf Acknowledgments.}
The authors thank Benjamin Matschke for the discussion and helpful remarks, thank the unknown referees for numerous remarks and comments, and thank Igor~Pak for useful advice on improving the readability of this text.

\section{Approximation of $\partial P$ by smooth surfaces}

We are going to use the following way to approximate a polytope by smooth bodies:

\begin{defn}
Denote by $P_\epsilon$ the union of all $\epsilon$-balls that are contained in $P$
$$
P_\epsilon = \bigcup_{B_\epsilon(x)\subseteq P} B_\epsilon.
$$
\end{defn}

The body $P_\epsilon$ has a smooth boundary and, by the known results, admits an inscribed regular octahedron. Moreover, in~\cite{karasev2009inscribing,karasev2010equipartitionbyfans} it is proved that there is a nontrivial $\mathbb Z_3$-equivariant $1$-homology class of such octahedra in the configuration space of all octahedra, for which we take the space of similarity transforms with positive determinant $S_3$, which is topologically equal to $\mathbb R^+\times \mathbb R^3\times SO(3)$. This space has an action of the symmetry group of the octahedron, and out of this group we are going to use the $\mathbb Z_3$ that cyclically permutes the three axes of the octahedron. 

In certain ``generic'' case the set $\mathcal I$ of inscribed octahedra is indeed a smooth $1$-manifold and itself represents a nonzero class $[\mathcal I] \in H_1 (S_3/\mathbb Z_3; \mathbb F_3)$. To handle degenerate cases (see~\cite{karasev2009inscribing,karasev2010equipartitionbyfans}) it is useful to replace $[\mathcal I]$ with its Poincar\'e dual class from some relative cohomology, which is actually convenient to treat as a member of the compact support cohomology $\zeta\in H^{9}_c(S_3/\mathbb Z_3; \mathbb F_3)$ (see~\cite{borel1960homology} as a reference for compact support cohomology). Using the technique of compact support cohomology, we restate the results of~\cite{karasev2009inscribing,karasev2010equipartitionbyfans} as follows: There exists a nonzero obstruction in cohomology $\zeta\in H^{9}_c(S_3/\mathbb Z_3; \mathbb F_3)$ such that the set of inscribed octahedra $\mathcal I$ has $\zeta$ supported in its arbitrarily small neighborhood even in the degenerate cases, when $\mathcal I$ itself does not look like a $1$-cycle geometrically.

From here on, we argue in terms of a $1$-dimensional homology cycle of $\mathcal I$ informally, having in mind that everything becomes rigorous when passing to codimension $1$ compact support cohomology.

Now let $\epsilon$ tend to zero. If the diameters of the inscribed octahedra of $P_\epsilon$ do not tend to zero, then we obtain an inscribed octahedron for $P$ by standard compactness considerations. Hence for the rest of the proof we assume the contrary: Let the maximum diameter of inscribed octahedra for~$P_\epsilon$ be at most $\delta(\epsilon)$ and $\displaystyle \lim_{\epsilon\to+0} \delta(\epsilon) = 0$.
Denote the set of inscribed octahedra for $P_\epsilon$ by $\mathcal I_\epsilon\subset S_3$.

Since $\delta(\epsilon)$ tends to zero, the octahedra from $\mathcal I_\epsilon$ tend (say, in Hausdorff metric) to points at the boundary of $P$. The first observation is that they obviously cannot tend to a relative interior point of a facet. Moreover, the detailed analysis near an edge shows that for small enough $\epsilon$ the octahedra in $\mathcal I_\epsilon$ cannot tend to an interior point of an edge. Indeed, if we project $\partial P$ along an edge then we obtain a plane angle $A$, the smoothening $\partial P_\epsilon$ being projected to a smoothened plane angle $A_\epsilon$. Let $C$ be an octahedron inscribed into $\partial P_\epsilon$; it projects to a quadrilateral or hexagon $C'$ inscribed into $\partial A_\epsilon$. Now it remains to note that $C'$ is centrally symmetric and we cannot see one of its vertices from any direction (of course, the invisible vertex depends on the direction), while we can see the entire $\partial A_\epsilon$ from some directions. This is a contradiction and we conclude:

\begin{lem}
If there is no octahedron inscribed into $\partial P$ then all octahedra in $\mathcal I_\epsilon$ tend to the vertices of $P$ in Hausdorff metric. For small enough $\epsilon$ the family $\mathcal I_\epsilon$ becomes a disjoint union of the sets $\mathcal I_\epsilon(v)$, corresponding to different vertices $v\in P$.
\end{lem}

The decomposition of $\mathcal I_\epsilon$ into the sets $\mathcal I_\epsilon(v)$ corresponds to decomposition of the $1$-homology class of inscribed octahedra into a smooth body into a sum of $1$-homology classes. Hence we have proved: 

\begin{lem}
There exists a vertex $v$, such that for arbitrarily small~$\epsilon$ the octahedra $\mathcal I_\epsilon(v)$ (inscribed into $\partial P_\epsilon$ near $v$) carry a nontrivial $\mathbb Z_3$-equivariant $1$-homology.
\end{lem}

Thus we have to study the situation near the vertices of $P$. More precisely, we have to consider solid angles $A(v)$ of corresponding vertices of $P$ and their smoothenings $A_\epsilon(v)$ (not depending on $\epsilon>0$ essentially because of the possibility to apply a homothety).

We want to describe the $\mathbb Z_3$-equivariant $1$-homology of the set of octahedra inscribed into $\partial A_\epsilon(v)$. The configuration space $S_3$ of all octahedra is homotopy equivalent to $SO(3)$ and it has trivial homology modulo $3$ in dimensions $1$ and $2$; the $\mathbb Z_3$-equivariant \mbox{$1$-homology} modulo $3$ (the homology of $S_3/\mathbb Z_3$) is therefore equal to $H_1(\mathbb Z_3; \mathbb F_3) = \mathbb F_3$. We choose the generator of $H_1(S_3/\mathbb Z_3; \mathbb F_3)$ to be equal to the $1$-homology of octahedra inscribed into a generic smooth convex body.

The homology class of octahedra inscribed into $\partial A_\epsilon(v)$ is well defined in the case when there are no arbitrarily large octahedra inscribed into $\partial A_\epsilon$ (note that we actually work in compact support cohomology and therefore must ensure that everything remains bounded); in this case it corresponds to the well defined $1$-homology class. Using an appropriate homothety we conclude that fixed $\epsilon$ and arbitrarily large octahedron is the same as arbitrarily small $\epsilon$ and a fixed size octahedron. By compactness considerations we conclude the following: 

\begin{lem}
The surface $\partial A_\epsilon(v)$ has no well defined homology class of inscribed octahedra only if there exists an octahedron inscribed into $\partial A(v)$ without smoothening.
\end{lem}

\begin{defn}
Call a solid angle $A$ \emph{special} if its boundary admits an inscribed octahedron.
\end{defn}

Denote the configuration space of all congruence classes ($SO(3)$-orbits) of solid angles (with $n\ge 3$ facets) by $\mathcal A^n$ and denote the subset of special angles by $\mathcal A^n_S$, this is a closed subset of $\mathcal A^n$. Define the function
$$
\phi : \mathcal A^n\setminus \mathcal A^n_S \to \mathbb F_3
$$
by assigning to a solid angle $A$ the $1$-homology of octahedra (divided by the generator of $H_1(S_3/\mathbb Z_3; \mathbb F_3)$) inscribed into the smoothened solid angle $A_\epsilon$.

\begin{lem}
The function $\phi$ is locally constant on $\mathcal A^n\setminus \mathcal A^n_S$.
\end{lem}

\begin{proof} 
Fix some $\epsilon>0$. Under deformations of a solid angle $A$ that avoid arbitrarily large inscribed octahedra for its smoothening $A_\epsilon$, the smoothenings $A_\epsilon$ are deformed continuously and the sizes of their inscribed octahedra remain bounded. Hence, by the continuity of the $1$-homology (or compact support cohomology), the class of octahedra inscribed into $A_\epsilon$ is preserved.
\end{proof}

Now in the inscribing problem we have the following options: 

a) $P$ has a special angle. In this case it already has an inscribed octahedron.

b) If (the sum is over all vertices in $P$)
\begin{equation}
\label{eq:hom-sum}
\sum_{v\in P} \phi (A(v)) \neq 1
\end{equation}
then $\partial P$ admits an inscribed octahedron. This follows from additivity of the homology classes.

Note that for every solid angle $A$ close enough to a halfspace $\phi(A)=0$, since its smoothening does not admit an inscribed octahedron.

\section{Proof of Theorem~\ref{thm:simple-inscr}}

If the polytope $P$ is simple then we deal with $\mathcal A^3_S\subset \mathcal A^3$. Taking into account the observations in the previous section, we see that to prove Theorem~\ref{thm:simple-inscr} it is enough to prove the following lemma (because in this case the left hand part of \eqref{eq:hom-sum} vanishes): 

\begin{lem}
\label{lem:non-sing-conn}
The set $\mathcal A^3\setminus \mathcal A^3_S$ is arcwise connected.
\end{lem}

The proof of Lemma~\ref{lem:non-sing-conn} will follow from the description of all solid angles that admit an octahedron inscribed into the boundary:

\begin{lem}
\label{lem:circumscirbed angle}
A solid angle $A$ admits an octahedron inscribed into its boundary if and only if it is possible to place its corresponding spherical triangle $v_1v_2v_3$ inside of the regular spherical triangle $t_1t_2t_3$ with $|t_1t_2|=\pi/3$ in the following way (up to relabeling the vertices $v_1$, $v_2$, $v_3$): The vertices $v_1$ and $t_1$ coincide, the vertex $v_2$ lies on the segment $t_1t_2$, and $v_3$ lies inside the triangle $t_1v_2t_3$. 
\end{lem}

\begin{center}
\includegraphics{figoct-3.mps}

\r\label{fig:triangle in triangle (3+2+1)}.
\end{center}

\begin{proof}

Let us check how an octahedron $C$ could be inscribed into $A$. There are two alternatives: 

Case~1: Some three vertices of $C$ are on one facet of $A$, two are on the other facet, and one is on the third facet.  

Case~2: Every facet of $A$ contains two vertices of $C$.

There are other degenerate cases, but we may ignore them since they all are limit cases of these two. 

Denote the vertices of the octahedron by $a$, $b$, $c$, $a'$, $b'$, and $c'$ (see Figure~\ref{fig:octahdron}).

\begin{center}
\includegraphics{figoct-1.mps}

\r\label{fig:octahdron}.
\end{center}

Consider the first case. Suppose one of the facets of the angle $A$ contains the facet~$a'b'c'$, the second facet contains the edge $ab$, and the third facet contains the vertex $c$.

Denote the vertex of the angle $A$ by $v$. It is easy to see that the common edge of the first and the second facets of $A$ is parallel to the edge $ab$ of $C$, because both facets are parallel to $ab$. Without loss of generality we may assume that the point $b$ is closer to $v$ than $a$. Let $v_1$ be the vertex of the spherical triangle that corresponds to the common edge of the first and the second facet of $A$, $v_2$ correspond to the first and the third facet, and $v_3$ is the remaining vertex (see Figure~\ref{fig:situation in 3d (3+2+1)}).

Let $t_1$, $t_2$, and $t_3$ be the vertices of the regular spherical triangle that corresponds to the vectors $\overrightarrow{ba}$, $\overrightarrow{c'b'}$, and~$\overrightarrow{a'c}$. 

\begin{center}
\includegraphics{figoct-2.mps}

\r\label{fig:situation in 3d (3+2+1)}.
\end{center}

Denote by $v'$ the point where the edge $v_3$ and the plane $abc$ intersect (Figure~\ref{fig:situation in 3d (3+2+1)}). Obviously, $v'c$ is parallel to the edge $v_2$.
Since $v'c$ does not intersect the interior of the triangle $abc$, it follows that the vector $\overrightarrow{v'c}$ lies ``between'' the vectors $\overrightarrow{ba}$ and $\overrightarrow{bc}$. Therefore on the sphere the vertex $v_2$ lies on the segment $[t_1,t_2]$.

Consider the plane $\alpha$ of the third facet that contains the point $c$ of the octahedron.
It contains the line~$cv'$ and does not intersect the octahedron in the interior.
This means that $\alpha$ lies ``between'' the planes (in terms of normals to oriented planes) $v'ca'$ and $v'ca$, the latter coinciding with the plane $abc$. The line $v_2t_3$ (on the sphere) corresponds to the plane $v'ca'$ and the line $t_1v_2$ corresponds to the plane $v'ca$.
Therefore the third facet corresponds to a line ``between'' $v_2t_1$ and $v_2t_3$.

It is clear that the second facet corresponds to the line passing through $t_1$ that lies ``between'' $[t_1,t_2)$ and $[t_1,t_3)$. Therefore the point $v_3$ lies inside the triangle $t_1v_2t_3$.

To prove the lemma in the opposite direction, we note that for any triangle from the statement of the lemma it is possible to construct an inscribed octahedron by the way depicted in Figure~\ref{fig:situation in 3d (3+2+1)}.

Consider the second case. Without loss of generality we may assume that the first facet of $A$ contains the edge $ab$ of the octahedron, the second facet contains the edge $ca'$, and the third facet contains the edge $b'c'$. This means that extensions of the sides of the spherical triangle $\triangle v_1v_2v_3$ pass through the vertices of $\triangle t_1t_2t_3$ (see Figure~\ref{fig:triangle in triangle (2+2+2)}). 
If the fact that $\triangle v_1v_2v_3$ is inside $\triangle t_1t_2t_3$ is not clear to the reader, please see the proof of Lemma~\ref{lem:poly-inclusion} for a rigorous explanation.

\begin{center}
\raisebox{-0.31cm}{\parbox{5cm}{
\begin{center}
\includegraphics{figoct-4.mps}

\r\label{fig:triangle in triangle (2+2+2)}.
\end{center}
}}
\parbox{6cm}{\begin{center}
\includegraphics{figoct-5.mps}

\r\label{fig:triangle in triangle (2+2+2) with another triangle}.
\end{center}}
\end{center}

Let us show that $\triangle v_1v_2v_3$ can be placed in $\triangle t_1t_2t_3$ in the proper way. Since the area of $\triangle v_1v_2v_3$ is less than the area of $\triangle t_1t_2t_3$ it follows that one of the angles of $\triangle v_1v_2v_3$ is less than $\angle t_2t_1t_3$ (note that $\triangle t_1t_2t_3$ is regular). Without loss of generality we may assume that this angle is $\angle v_2$ and the triangle $v_1v_2v_3$ is placed in $t_1t_2t_3$ in the way shown in Figure~\ref{fig:triangle in triangle (2+2+2)} (the points $t_1$ and $v_3$ are on the same side of the line $v_1v_2$).

Choose a point $t'_1$ so that $|t_1'v_2|=|t_1t_2|$ and $\angle t'_1v_2t_3=\angle t_1t_2t_3$; and choose a point~$t'_3$ on the ray $[v_2, t_3)$ so that $|t'_3v_2|=|t_1t_2|$ (Figure~\ref{fig:triangle in triangle (2+2+2) with another triangle}). Note that $\angle t_3v_2t_1>\angle t_3t_2t_1$. Therefore the segment $[t'_1, v_2]$ intersects the segment $[t_1, v_3]$, thus giving the inclusion $\triangle v_1v_2v_3 \subset \triangle t'_1v_1v_2$. We obtain that $\triangle v_1v_2v_3$ is positioned in the proper way inside $\triangle v_2t'_1t'_3$, which is congruent to $\triangle t_1t_2t_3$.
\end{proof}

\begin{lem}
\label{lem:circumscribed angle tight}
In Lemma~\ref{lem:circumscirbed angle} we may assume that $|v_1v_2|\ge |v_1v_3|\ge |v_2v_3|$.
\end{lem}

\begin{proof}[Proof of Lemma \ref{lem:circumscribed angle tight}]
If $|v_1v_3|>|v_1v_2|$ then reflect $\triangle v_1v_2v_3$ with respect to the bisector of $\angle v_2v_1v_3$ and the triangle $v_1v_2'v_3'$, which lies in $\triangle t_1t_2t_3$ in a proper way, because $v_2'\in [v_1,v_3]$ and $\triangle v_1v_2t_3\subset v_1v_3't_3$.

If $|v_2v_3|>|v_1v_3|$ then reflect $\triangle v_1v_2v_3$ with respect to the perpendicular bisector of the segment $[v_1,v_2]$. Denote by $v_3'$ the image of $v_3$. We have 
\[\angle v_3'v_2v_1=\angle v_3v_1v_2<\angle t_3v_1v_2<\angle t_3v_2v_1.\]
Since  $|v_2v_3|>|v_1v_3|$, we have 
\[\angle v_3'v_1v_2=\angle v_3v_2v_1<\angle v_3v_1v_2<\angle t_3v_1v_2.\]
Therefore the ``rays'' $[v_1, v_3')$ and $[v_2, v_3')$ are directed into the interior of $\triangle t_1v_2t_3$ and the point $v_3'$ lies inside this triangle.

Using this two kinds of operations we can rearrange the side lengths of the $\triangle v_1v_2v_3$ in the required order.
\end{proof}

\begin{cor}
\label{cor:small angle}
If all facet angles of a solid angle $A\in \mathcal A^3$ are less than $\pi/6$ then $A\in \mathcal A^3_S$. If one facet angle of a solid angle $A\in \mathcal A^3$ is greater than $\pi/3$ then $A\in \mathcal A^3\setminus \mathcal A^3_S$.
\end{cor}

\begin{rem}
It is \emph{not true} that if all facet angles of a solid angle $A\in \mathcal A^3$ are at most  $\pi/3$ then $A\in \mathcal A^3_S$. Indeed, we start from the spherical triangle $t_1t_2t_3$ and decrease slightly its side $t_2t_3$. The resulting triangle will not fit into $\triangle t_1t_2t_3$, so by Lemma~\ref{lem:circumscirbed angle} it cannot be in $\mathcal A^3_S$.
\end{rem}

Now we make the final step:

\begin{proof}[Proof of Lemma~\ref{lem:non-sing-conn}]
Consider the triangle $T$ corresponding to a solid angle $A\in \mathcal A^3\setminus \mathcal A^3_S$. Let $T=\triangle v_1v_2v_3$ and $|v_1v_2|\ge |v_1v_3|\ge |v_2v_3|$. Let $T_0$ be the regular triangle $t_1t_2t_3$ with side length $\pi/3$.

We are going to show how to increase the sides of the triangle $T$ and obtain a triangle with a side greater than $\pi/3$ (the set of triangles of this kind is obviously arcwise connected and by Corollary~\ref{cor:small angle} belongs to the set $\mathcal A^3\setminus \mathcal A^3_S$).

Suppose all sides are less than $\pi/3$. Let us try to place the triangle $T$ into the regular triangle $T_0$ in the way prescribed by Lemma~\ref{lem:circumscribed angle tight}. The only way that makes the position not proper is that $v_3$ is outside $\triangle v_1v_2t_3$, which is possible only if the segment $v_1v_3$ goes outside the segment $v_2t_3$. From Corollary~\ref{cor:small angle} it follows that $|v_1v_2| \ge \pi/6$ and therefore $\angle v_1v_2t_3$ is acute. Now we increase the length of the side $v_1v_3$ up to $|v_1v_2|$ preserving the angle $\angle v_2v_1v_3$. The position of $\triangle v_1v_2v_3$ remains not proper (as required by Lemma~\ref{lem:circumscribed angle tight}). 


Now we start to increase the length of the sides $v_1v_2$ and $v_1v_3$ in such a way that they remain equal during the process. The angle $\angle v_1v_2v_3$ will increase while the angle $\angle v_1v_2t_3$ will decrease during this process. Hence the point $v_3$ will remain outside $\triangle v_1v_2t_3$. Finally we obtain an isosceles triangle with two sides greater than $\pi/3$.

\begin{center}
\includegraphics{figoct-7.mps}\hskip 0.3cm
\includegraphics{figoct-8.mps}\hskip 0.3cm
\includegraphics{figoct-9.mps}\hskip 0.3cm
\includegraphics{figoct-10.mps}

\r
\end{center}
\end{proof}

\section{The case of non-simple polytopes}
\label{sec:non-simple-sec}

In this case we have a weaker analogue of Lemma~\ref{lem:circumscirbed angle}. We again associate a solid angle~$A$ with its spheric convex polygon and denote by $T_0$ the regular spherical triangle with side length $\pi/3$.

\begin{lem}
\label{lem:poly-inclusion}
If $A\in \mathcal \mathcal A_S$ then some congruence takes the spherical polygon of $A$ inside~$T_0$.
\end{lem}

\begin{proof}
The assumption $A\in\mathcal A_S$ means that the vertex $v$ of $A$ is outside $C$ and we ``see'' all vertices of $C$ from $v$. Here we \emph{see} a vertex $c\in C$ if the segment $[vc]$ does not intersect the interior of $C$. We also note that it is sufficient to prove the lemma in the ``generic'' case and then go to the limit using compactness of the group of rotations. So we are free to perturb anything.

Consider a regular tetrahedron $\Theta$ formed by some four alternating facets of $C$. Observe that $v$ is not in the interior of $\Theta$, otherwise some vertex of $C$ would be ``behind'' $C$ when looking from $v$. Look at $\Theta$ from $v$, there are three alternatives:

Case~I: We see some vertex of $\Theta$ and this vertex is not ``behind'' $\Theta$. Denote by $B$ the solid angle of this vertex, its spherical triangle is congruent to $T_0$. Let us make a small perturbation of $A$ keeping $C$ inscribed into $A$ and making the intersection $\partial A \cap \partial B$ transversal.

Let $\partial B$ consist of three flat angles $B_1$, $B_2$, and $B_3$. Every intersection $X_i = A\cap B_i$ is a convex set containing the vertex $s$ of $B_i$ and having three vertices $x_1$, $x_2$, and $x_3$ of the octahedron $C$ on its boundary. Consider the facet $B_1$, in this case $x_1$, $x_2$, and $x_3$ is the triple of verticies $a$, $b'$, and $c$ (Figure~\ref{fig:circumscriped polygon}).

Note that $a$ and $c$ are on the sides of $B_1$ and the segments $[a,b']$ and $[b',c]$ are parallel to the sides of $B_1$, so $sab'c$ is a parallelogram. There exists a support line $\ell_1$ to $X_1$ passing through $b'$ in the plane of $B_1$. The points $a$, $c$, and $s$ are on the one side of $\ell_1$ and the line~$\ell_1$ separates $X_1$ from infinity, except for the case when $\ell_1$ is parallel to a side of~$B_1$. But the latter situation is degenerate and can be excluded by a small perturbation. A similar statement holds for facets $B_2$ and $B_3$.

Thus $\partial A\cap \partial B$ is bounded and after the translation that identifies the vertices of $A$ and $B$ the whole solid angle $A$ will get inside $B$. This is true after an arbitrarily small perturbation of $A$, so it was true for the original $A$ by the continuity.

\begin{center}
\includegraphics{figoct-6.mps}

\r\label{fig:circumscriped polygon}.
\end{center}

Case~II: We see some vertex of $\Theta$ and this vertex is ``behind'' $\Theta$. In this case we replace $\Theta$ with another tetrahedron $\Theta'$ symmetric to $\Theta$ with respect to the center of $C$. The considered case cannot happen to both $\Theta$ and $\Theta'$ at the same time.

Case~III: We see $\Theta$ as a quadrangle from $v$. In this case we do not see the vertex of $C$ that corresponds to the farthest in the pair of edges of $\Theta$ that intersect as we see them from $v$, because we see every face of $C$ adjacent to $v$ from inside.
\end{proof}

Now we make a definition:

\begin{defn}
Denote by $\mathcal A_0\subset \mathcal A$ the set of solid angles that cannot be put into $T_0$ by a~congruence.
\end{defn}

\begin{rem}
A careful analysis of Lemmas~\ref{lem:circumscirbed angle} and~\ref{lem:circumscribed angle tight} shows that $\mathcal A^3\cap \mathcal A_0\neq \mathcal A^3\setminus \mathcal A^3_S$. It~is sufficient to take $T=T_0$ and shrink one of its sides to the midpoint of that side slightly.
\end{rem}

\begin{thm}
\label{thm:non-thm:simple-inscr}
Suppose $P$ is a convex polytope in $\mathbb R^3$ such that all its solid angles either in $\mathcal A^3$ or in $\mathcal A_0$. Then there exists a regular octahedron inscribed into $\partial P$.
\end{thm}

\begin{proof}
It remains to show that any $A\in \mathcal A_0$ can be deformed to a halfspace staying inside $\mathcal A_0$. We can make a strong monotonic (monotonic with respect to inclusion) deformation of $A$ to a halfspace, and obviously $A$ will remain in $\mathcal A_0$ under such a deformation.
\end{proof}




\nocite{karasev2010knaster}
\end{document}